\documentclass[10pt]{amsart}

\usepackage{amsfonts}
\usepackage{amsmath}
\usepackage{amssymb}
\usepackage{amsthm}
\usepackage{enumerate}
\usepackage{mdwlist}
\usepackage{mathrsfs}
\usepackage{mathabx}
\usepackage{url}
\usepackage{graphicx}
\usepackage{latexsym}
\usepackage{stmaryrd}
\usepackage{verbatim}
\usepackage{xcolor}

\newtheorem{theorem}{Theorem}[section]
\newtheorem{proposition}[theorem]{Proposition}
\newtheorem{lemma}[theorem]{Lemma}
\newtheorem{corollary}[theorem]{Corollary}

\theoremstyle{remark}
\newtheorem{definition}[theorem]{Definition}

\usepackage{mathtools}
\DeclarePairedDelimiter{\ceil}{\lceil}{\rceil}

\newcommand{\NN}{\mathbb{N}}
\newcommand{\RR}{\mathbb{R}}
\newcommand{\PP}{\mathbb{P}}
\newcommand{\EE}{\mathbb{E}}

\newcommand{\norm}[1]{\lVert {#1} \rVert}
\newcommand{\seq}[1]{({#1})}

\begin{document}

\title[A quantitative Robbins-Siegmund theorem]{A quantitative Robbins-Siegmund theorem}

\author[Morenikeji Neri and Thomas Powell]{Morenikeji Neri and Thomas Powell}
\date{\today}

\begin{abstract}
The Robbins-Siegmund theorem is one of the most important results in stochastic optimization, where it is widely used to prove the convergence of stochastic algorithms. We provide a quantitative version of the theorem, establishing a bound on how far one needs to look in order to locate a region of \emph{metastability} in the sense of Tao. Our proof involves a metastable analogue of Doob's theorem for $L_1$-supermartingales along with a series of technical lemmas that make precise how quantitative information propagates through sums and products of stochastic processes. In this way, our paper establishes a general methodology for finding metastable bounds for stochastic processes that can be reduced to supermartingales, and therefore for obtaining quantitative convergence information across a broad class of stochastic algorithms whose convergence proof relies on some variation of the Robbins-Siegmund theorem. We conclude by discussing how our general quantitative result might be used in practice.
\end{abstract}

\maketitle

\section{Introduction}
\label{sec:intro}

Abstract convergence results pertaining to almost-supermartingales are of fundamental importance in stochastic approximation. The most well-known and widely used is the so-called Robbins-Siegmund theorem:
\begin{theorem}[Robbins-Siegmund \cite{robbins-siegmund:71:lemma}]
\label{res:RS:original}
Let $\seq{X_n}$, $\seq{A_n}$, $\seq{B_n}$ and $\seq{C_n}$ be sequences of nonnegative integrable random variables on some arbitrary probability space and adapted to the filtration $\seq{\mathcal F_n}$, with $\sum_{i=0}^\infty A_i<\infty$ and $\sum_{i=0}^\infty C_i<\infty$ almost surely and 
\[
\EE[X_{n+1}\mid\mathcal F_n]\leq (1+A_n)X_n-B_n+C_n
\]
almost surely for all $n\in\NN$. Then, almost surely, $\seq{X_n}$ converges and $\sum_{i=0}^\infty B_i<\infty$.
\end{theorem}
Theorem \ref{res:RS:original} has been used to establish the convergence of a broad range of classic stochastic approximation algorithms, most famously those based on the Robbins-Monro scheme \cite{robbins-monro:51:stochastic} such as stochastic gradient descent, and it is generally applicable to any kind of stochastic process that enjoys a \emph{quasi-Fej\'er monotonicity property}, of which there are many examples (cf. recent developments along these lines in \cite{combettes-pesquet:15:fejer} or the survey \cite{franci-grammatico:convergence:survey:22}).

Direct, numerical information for Theorem \ref{res:RS:original} in the form of convergence rates for $\seq{X_n}$ and $\sum_{i=0}^\infty B_i<\infty$ are in general not possible, as the result implies Doob's $L_1$-convergence theorem for supermartingales and thus convergence can be arbitrarily slow. As such, most existing works on convergence rates for algorithms reliant on the Robbins-Siegmund theorem (recent examples include \cite{karandikar:pp:stochastic,liu:etal:22:RSrates,neri-pischke-powell:pp:fejer,sebbouh:etal:21:RSrates}) focus on specific instances of the theorem where additional assumptions are imposed on the negative term $-B_n$ to ensure a strong descent and thus convergence rates. 

We consider quantitative supermartingale convergence from a different angle, and instead provide a quantitative version of the full Robbins-Siegmund theorem in terms of \emph{metastability} (in the sense of Tao \cite{tao:07:softanalysis}). Our approach is inspired by the logic-based \emph{proof mining} program \cite{kohlenbach:08:book}, that uses insights and techniques from logic to obtain both quantitative and qualitative improvements of theorems in mainstream mathematics through the analysis of their proofs. In the last few years, techniques from proof mining have been brought to bear on probability theory for the very first time, including an abstract framework for explaining the relevant logical aspects of probability \cite{neri-pischke:pp:formal}, along with a quantitative study of martingales \cite{neri-powell:25:martingale}, where the latter is of particular relevance for what follows.

This paper then represents one of the first applications of proof theory to stochastic optimization, and we envisage our main result being further developed and applied to obtain concrete numerical information from specific proofs that use the Robbins-Siegmund theorem. Indeed, an important benefit of the logical approach is that it works in a modular way, where informally speaking, as long as a theorem has been suitably finitized, bounds for that finitary theorem can be ``plugged in'' directly to proofs that use it as a lemma, and in this way used as a component in a broader analysis (see \cite{kohlenbach:08:book} for a comprehensive account of this phenomenon). Examples relating to convergence -- where a metastable convergence theorem has been used to obtain direct numerical information -- abound in the proof mining literature, though until now always in a nonstochastic setting. In Section \ref{sec:RS:app} we provide a case study exactly along these lines but for stochastic algorithms, in which our quantitative Robbins-Siegmund theorem is used to bound the number of iterations of an abstract stochastic algorithm until an approximate solution is found. Given the widespread use of the Robbins-Siegmund theorem across stochastic optimization, we anticipate that our main result can be applied in many different ways to obtain new numerical information for stochastic algorithms. 

\section{Quantitative stochastic convergence}
\label{sec:convergence}

We begin by collecting together a number of preliminary definitions and lemmas. We use the notation $[n;m]:=\{n,n+1,\ldots,m-1,m\}$ with the convention that $[n;m]=\emptyset$ if $m<n$. We freely use the logical quantifier symbols $\exists,\forall$ throughout as they allow us to express several of our main definitions in a compact manner, and whenever these occur inside a probability measure they implicitly refer to an appropriate equivalent formulation in terms of infinite unions or intersections.

Let $\seq{x_n}$ be a sequence of reals, which is convergent iff it satisfies the Cauchy property:
\begin{equation}
\label{eqn:cauchy}
\forall\varepsilon>0\, \exists n\, \forall i,j\geq n\left(|x_i-x_j|<\varepsilon\right).
\end{equation}
The most direct piece of quantitative information we could hope for in this context is an explicit function $f:(0,1)\to\NN$ such that for each $\varepsilon>0$ we have $|x_i-x_j|<\varepsilon$ for all $i,j\geq f(\varepsilon)$. However, for many classes of convergent sequences, an explicit convergence rate is not possible, so instead we consider the following alternative characterisation of convergence, so-called \emph{metastable} convergence, that in general \emph{can} be characterised computationally:
\begin{equation}
\label{eqn:metastable}
\forall\varepsilon>0\, \forall g:\NN\to\NN\, \exists n\, \forall i,j\in [n;n+g(n)]\left(|x_i-x_j|<\varepsilon\right).
\end{equation}
To see that (\ref{eqn:cauchy}) and (\ref{eqn:metastable}) are equivalent, note that (\ref{eqn:cauchy}) clearly implies (\ref{eqn:metastable}), and conversely, if (\ref{eqn:cauchy}) doesn't hold then for some $\varepsilon>0$ there exists some $g$ such that $$\forall n\, \exists i,j\in [n;n+g(n)](|x_i-x_j|\geq\varepsilon)$$ and taking the negation of this we obtain (\ref{eqn:metastable}). A functional $\Phi:(0,1)\times \NN^\NN\to\NN$ satisfying
\begin{equation*}
\label{eqn:metastable:rate}
\forall\varepsilon>0\, \forall g:\NN\to\NN\, \exists n\leq\Phi(\varepsilon,g)\, \forall i,j\in [n;n+g(n)]\left(|x_i-x_j|<\varepsilon\right)
\end{equation*}
is called a \emph{metastable rate of convergence} (or just \emph{rate of metastability}) for $\seq{x_n}$. A detailed account of metastabile convergence and its connection to formal logical transformations can be found in \cite{kohlenbach:08:book} (see Chapter 2 in particular), while its significance in mainstream mathematics is discussed in \cite{tao:07:softanalysis}. In contrast to direct rates of convergence, explicit and highly-uniform metastable rates of convergence for classes of sequences can often be obtained through analysing the respective convergence proofs, and indeed, the extraction of rates of metastability is a standard result in proof mining with numerous examples in recent years.

A very natural situation in which convergent sequences posses particularly simple metastable rates are those which for which an explicit bound on their $\varepsilon$-fluctuations can be given. We say that a sequence $\seq{x_n}$ experiences $k$ $\varepsilon$-fluctuations if there exists
\begin{equation*}
i_1<j_1\leq i_2<j_2\leq\ldots \leq i_k<j_k \mbox{ with  $|x_{i_l}-x_{j_l}|\geq \varepsilon$ for all $l=1,\ldots,k$.}
\end{equation*}
If $\phi:(0,1)\to \RR$ is a function such that for any $\varepsilon>0$, $\seq{x_n}$ has no more than $\phi(\varepsilon)$ $\varepsilon$-fluctuations, then the following is a metastable rate of convergence for $\seq{x_n}$: 
\begin{equation}
\label{eqn:learnable}
\Phi(\varepsilon,g):=\tilde g^{(\lceil\phi(\varepsilon)\rceil)}(0)
\end{equation}
where $\tilde g(n):=n+g(n)$ and $\tilde g^{(i)}$ denotes the $i$th iteration of $\tilde g$. This can be seen by observing that at least one of the intervals $[\tilde g^{(n)}(0);\tilde g^{(n+1)}(0)]$ for $n\leq \lceil\phi(\varepsilon)\rceil$ cannot contain any fluctuation. Indeed, as highlighted in \cite[Section 2.3]{neri-powell:25:martingale}, when a sequence enjoys a metastable rate of convergence of the simple form (\ref{eqn:learnable}), this is equivalent $\phi$ having the following property:
\begin{definition}
\label{def:learnable}
Let $\seq{x_n}$ be a sequence of real numbers. Any function $\phi:(0,1)\to \RR$ such that for all $\varepsilon\in (0,1)$, for any increasing sequences $a_0<b_0\leq a_1<b_1\leq \ldots$ there exists $n\leq \phi(\varepsilon)$ such that
\[
\forall i,j\in [a_n;b_n]\, (|x_i-x_j|< \varepsilon)
\]
is called a \emph{learnable rate of convergence} for $\seq{x_n}$.
\end{definition}
There are many situations in which convergent sequences possess learnable rates of convergence, and hence simple metastable rates of the form (\ref{eqn:learnable}). The following easily-proven result is standard in proof mining (see in particular \cite[Chapter 2]{kohlenbach:08:book}), and is also discussed by Tao in \cite{tao:07:softanalysis} where it is labelled the `finite convergence principle':
\begin{lemma}
\label{res:monotone}
Let $\seq{x_n}$ be a nondecreasing sequence nonnegative reals such that $x_n<K$ for all $n\in\NN$. Then a learnable rate of convergence for $\seq{x_n}$ is given by $\phi(\varepsilon):=K/\varepsilon$, which in turn implies that a metastable rate of convergence is given by $\Phi(\varepsilon,g):=\tilde g^{\lceil K/\varepsilon\rceil}(0)$.  
\end{lemma}
%
%
For example, in the special case of a convergent series $\sum_{i=0}^\infty \alpha_i<\infty$ where $\seq{\alpha_n}$ is some sequence of nonnegative reals, if $K>0$ is any number satisfying $\sum_{i=0}^\infty \alpha_i< K$ then for any $a_0<b_0\leq a_1<b_1\leq \ldots$ then $\phi(\varepsilon):=K/\varepsilon$ is a learnable rate of convergence for $\seq{\sum_{i=n}^\infty \alpha_i}$.

Metastable rates of the form (\ref{eqn:learnable}) represent a simple instance of a more general kind of iterative metastable rate that arises from formulas that are \emph{effective learnability}, a concept introduced and explored in \cite{kohlenbach-safarik:14:fluctuations}, though for the purposes of this paper it is sufficient to restrict our attention to learnable rates of convergence in the simple sense of Definition \ref{def:learnable} above. A notable advantage of having learnables rates (in our sense) is that they compose very easily under addition and multiplication of sequences, as will be shown in Lemma \ref{res:addtimes:stochastic} below.

We now introduce stochastic analogues of each of the concepts discussed above, where sequences of real numbers are replaced by sequences of random variables $\seq{X_n}$ over some probability space $(\Omega,\mathcal{F},\PP)$, and we now talk about properties being true \emph{almost-surely}. Informally speaking, we arrive at suitable quantitative notions of these properties by first using continuity properties of the probability function $\PP$ to bring quantifiers from inside the measure to the outside: See in particular \cite[Section 3]{neri-powell:25:martingale} for a detailed discussion of this process of `finitization'. We first note that almost-sure Cauchy convergence of $\seq{X_n}$ i.e.
\[
\PP\left(\forall \varepsilon>0\, \exists n\, \forall i,j\geq n\left(|X_i-X_j|<\varepsilon\right)\right)=1
\]
is equivalent to the following metastable variant (introduced for general measures in \cite{avigad-dean-rute:12:dominated}):
\[
\forall \lambda,\varepsilon>0\, \forall g:\NN\to\NN\, \exists n\, \left[\PP\left(\exists i,j\in [n;n+g(n)]\left(|X_i-X_j|\geq \varepsilon\right)\right)<\lambda\right].
\]
This then leads to the following definition:
\begin{definition}[\cite{avigad-dean-rute:12:dominated,avigad-gerhardy-towsner:10:local,neri-powell:25:martingale}]
\label{def:metastability:uniform}
Let $\seq{X_n}$ be a stochastic process. Any functional $\Phi:(0,1)\times (0,1)\times \NN^\NN\to \NN$ such that for all $\lambda,\varepsilon\in (0,1)$ and $g:\NN\to\NN$ there exists $n\leq \Phi(\lambda,\varepsilon,g)$ satisfying
\begin{equation*}
\PP(\exists i,j\in [n;n+g(n)](|X_i-X_j|\geq \varepsilon))<\lambda
\end{equation*}
is called a \emph{metastable rate of uniform convergence} for $\seq{X_n}$.
\end{definition}
Fluctuations for stochastic processes have been widely studied (for martingales and ergodic averages, \cite{kachurovskii:96:convergence} is particularly relevant), and just as in the deterministic case, bounds on fluctuations give rise to simple iterative rates of metastability. Here the relevant stochastic analogue of learnable convergence (one of many possible notions of stochastic learnability as discussed in \cite{neri-pischke-powell:25:learnability}) is the following:
\begin{definition}[\cite{neri-powell:25:martingale}]
\label{def:learnable:uniform}
Let $\seq{X_n}$ be a stochastic process. Any function $\phi:(0,1)\times (0,1)\to \RR$ such that for all $\lambda,\varepsilon\in (0,1)$, for any increasing sequences $
a_0<b_0\leq a_1<b_1\leq \ldots$ there exists $n\leq \phi(\lambda,\varepsilon)$ such that
\[
\PP\left(\exists i,j\in [a_n;b_n]\, (|X_i-X_j|\geq \varepsilon)\right)<\lambda
\]
is called a \emph{learnable rate of uniform convergence} for $\seq{X_n}$.
\end{definition}
The connection between learnable rates in the above sense and fluctuations is as follows:
\begin{theorem}
[\cite{neri-powell:25:martingale}]
Let $\seq{X_n}$ be a stochastic process and let $J_\varepsilon\seq{X_n}$ denote the maximum number of $\varepsilon$-fluctuations experienced by $\seq{X_n}$. Then
\[
\phi(\lambda,\varepsilon):=\frac{\EE\left[J_\varepsilon\seq{X_n}\right]}{\lambda}
\]
is a learnable rate of uniform convergence for $\seq{X_n}$.
\end{theorem}
As such, learnable rates of uniform convergence are closely related to and generalise the property of a stochastic process having bounded fluctuations in mean. Entirely analogously to the deterministic case, whenever $\phi$ is a learnable rate of uniform convergence for $\seq{X_n}$, $$\Phi(\varepsilon,\lambda,g):=\tilde g^{(\lceil \phi(\lambda,\varepsilon)\rceil)}(0)$$ is a metastable rate of uniform convergence for $\seq{X_n}$ (for details see \cite{neri-powell:25:martingale}).

In order to deal with infinite series along with sums and products of stochastic processes in a quantitative way, we must first give a quantitative interpretation to the statement that $\sup_{n\in\NN}|X_n|<\infty$ almost surely. Again, bringing out the relevant quantifiers from within the probability measure function gives us the following notion:
\begin{definition}
\label{def:uniform:boundedness}
Let $\seq{X_n}$ be a stochastic process. Any function $\rho:(0,1)\to \RR$ with
\[
\PP\left(\sup_{n\in\NN}|X_n|\geq \rho(\lambda)\right)<\lambda
\]
for all $\lambda\in (0,1)$ will be called a \emph{modulus of uniform boundedness} for $\seq{X_n}$. 
\end{definition}
The following result on convergence for almost sure monotone sequences is important in what follows, and forms a stochastic analogue of Lemma \ref{res:monotone} above.
\begin{lemma}
\label{res:monotone:stochastic}
Let $\seq{X_n}$ be a nonnegative stochastic process such that $X_n\leq X_{n+1}$ pointwise for all $n\in\NN$, and suppose that $\rho:(0,1)\to\RR$ is a modulus of uniform boundedness for $\seq{X_n}$. Then
\[
\phi(\lambda,\varepsilon):=\frac{2\cdot \rho(\lambda/2)}{\lambda\varepsilon}
\]
is a learnable rate of uniform convergence for $\seq{X_n}$.
\end{lemma}

\begin{proof}
Fix $\lambda,\varepsilon\in (0,1)$ and $a_0<b_0\leq a_1<b_1\leq\ldots$. Then for any $n\in\NN$ we have
\begin{equation*}
\begin{aligned}
\PP\left(\exists i,j\in [a_n;b_n]\, |X_i-X_j|\geq \varepsilon\right)&=\PP\left(X_{b_n}-X_{a_n}\geq \varepsilon\right)\\
&\leq \PP\left(X_{b_n}-X_{a_n}\geq \varepsilon\cap \sup_{n\in\NN}X_n<\rho(\lambda/2)\right)+\frac{\lambda}{2}
\end{aligned}
\end{equation*}
Setting $Q_n:=X_{b_n}-X_{a_n}\geq \varepsilon$ and $R:=\sup_{n\in\NN}X_n<\rho(\lambda/2)$ it suffices to show that there exists some $n\leq \phi(\lambda,\varepsilon)$ such that $\PP(Q_n\cap R)<\lambda/2$. If this were not the case we would have
\[
\frac{\rho(\lambda/2)}{\varepsilon}<\frac{\lambda}{2}\left(\phi(\lambda,\varepsilon)+1\right)\leq \sum_{n=0}^{\phi(\lambda,\varepsilon)} \PP(Q_n\cap R)=\EE\left[I_R\sum_{n=0}^{\phi(\lambda,\varepsilon)} I_{Q_n}\right]
\]
But for $\omega \in R$ we can only have $\omega \in Q_n$ for strictly fewer than $\rho(\lambda/2)/\varepsilon$ distinct values of $n\in\NN$, since whenever $\omega\in Q_{n_1}\cap \ldots \cap Q_{n_k}\cap R$ for $n_1<\ldots<n_k$ then
\[
\rho(\lambda/2)> X_{b_{n_k}}(\omega)\geq \sum_{j=1}^k \left(X_{b_{n_j}}(\omega)-X_{a_{n_j}}(\omega)\right)\geq k\varepsilon
\]
Therefore we obtain
\[
\EE\left[I_R\sum_{n=0}^{\phi(\lambda,\varepsilon)} I_{Q_n}\right]\leq\EE\left[I_R\cdot\frac{\rho(\lambda/2)}{\varepsilon}\right]\leq \frac{\rho(\lambda/2)}{\varepsilon}
\]
which is a contradiction.\end{proof}
In particular, almost sure convergence of a series $\sum_{i=0}^\infty A_i$ where now the $A_i$ are nonnegative random variables can be represented quantitatively through a modulus of uniform boundedness $\rho$ i.e. a function satisfying
\[
\PP\left(\sum_{i=0}^\infty A_i\geq \rho(\lambda)\right)<\lambda
\] 
for all $\lambda\in (0,1)$, and this will be our preferred definition of quantitative convergence for infinite series. It is related to various other quantitative notions of convergence as follows:

\begin{proposition}
\label{res:sums:converge}
Let $\seq{A_n}$ be a sequence of nonnegative random variables.

\begin{enumerate}[(a)]

	\item If $\phi:(0,1)^2\to \NN$ is a direct rate of convergence for $\sum_{i=0}^\infty A_i$ in the sense that
	\[
	\PP\left(\sum_{i=\phi(\lambda,\varepsilon)}^\infty A_i\geq \varepsilon\right)<\lambda
	\]
	for all $\lambda,\varepsilon\in (0,1)$, and $a>0$ is such that $A_n\subseteq [0,a]$ for all $n\in\NN$, then for any $\varepsilon\in (0,1)$
	\[
	\rho_\varepsilon(\lambda):=a\cdot \phi(\lambda,\varepsilon)+\varepsilon
	\]
	is a modulus of uniform boundedness for $\sum_{i=0}^\infty A_i$.\smallskip
	
	\item If $\rho:(0,1)\to \NN$ is a modulus of uniform boundedness for $\sum_{i=0}^\infty A_i$ then
	\[
	\psi(\lambda,\varepsilon):=\frac{2\cdot \rho(\lambda/2)}{\lambda\varepsilon}
	\]
	is a learnable rate of uniform convergence for $\seq{\sum_{i=0}^n A_i}$.

\end{enumerate}
\end{proposition}

\begin{proof}
For part (a) we observe that
\begin{equation*}
\begin{aligned}
\PP\left(\sum_{i=0}^\infty A_i\geq a\cdot \phi(\lambda,\varepsilon)+\varepsilon\right)&\leq \PP\left(\sum_{i=0}^\infty A_i\geq \sum_{i=0}^{\phi(\lambda,\varepsilon)-1}A_i+\varepsilon\right)\\
&= \PP\left(\sum_{i=\phi(\lambda,\varepsilon)}^\infty A_i\geq \varepsilon\right)<\lambda
\end{aligned}
\end{equation*}
Part (b) follows from Lemma \ref{res:monotone:stochastic} since $\seq{\sum_{i=0}^n A_i}$ is monotone increasing.
\end{proof}

We conclude by showing how both learnable rates and moduli of uniform boundedness combine under arithmetic operations:

\begin{lemma}
\label{res:addtimes:stochastic}
Suppose that $\seq{X_n}$ and $\seq{Y_n}$ converge with learnable rates of uniform convergence $\phi$ and $\psi$, respectively. Then
\begin{enumerate}[(a)]

	\item a learnable rate of uniform convergence for $\seq{X_n+Y_n}$ is given by $$\chi(\lambda,\varepsilon):=\phi(\lambda/2,\varepsilon/2)+\psi(\lambda/2,\varepsilon/2)$$
	
	\item a learnable rate of uniform convergence for $\seq{X_nY_n}$ is given by $$\chi(\lambda,\varepsilon):=\phi\left(\frac{\lambda}{4},\frac{\varepsilon}{2\sigma(\lambda/4)}\right)+\psi\left(\frac{\lambda}{4},\frac{\varepsilon}{2\rho(\lambda/4)}\right)$$ where $\rho$ and $\sigma$ are moduli of uniform boundedness for $\seq{X_n}$ and $\seq{Y_n}$, respectively.

\end{enumerate}
\end{lemma}

\begin{proof}
We only prove (b), as part (a) is analogous but simpler (and in any case proven as Lemma 7.3 of \cite{neri-powell:25:martingale}). Fix $\lambda,\varepsilon\in (0,1)$and define the events
\[
Q:=\sup_{n\in\NN}|X_n|<\rho\left(\frac{\lambda}{4}\right) \ \ \ \mbox{and} \ \ \ R:=\sup_{n\in\NN}|Y_n|<\sigma\left(\frac{\lambda}{4}\right) 
\]
Take any $a_0<b_0\leq a_1<b_1\leq \ldots$. Then for any $n\in\NN$ we have
\begin{equation*}
\begin{aligned}
&\PP\left(\exists i,j\in [a_n;b_n]\left(\vert X_iY_i-X_jY_j\vert\geq \varepsilon\right)\right)\\
&\leq \PP\left(\exists i,j\in [a_n;b_n]\left(\vert X_iY_i-X_jY_j\vert\geq \varepsilon\right)\cap R\cap Q\right)+\frac{\lambda}{2}
\end{aligned}
\end{equation*}
so it suffices to show that there exists some $n\leq \chi(\lambda,\varepsilon)$ such that
\begin{equation}
\label{eqn:prod}
\PP\left(\exists i,j\in [a_n;b_n]\left(\vert X_iY_i-X_jY_j\vert\geq \varepsilon\right)\cap R\cap Q\right)<\frac{\lambda}{2}.
\end{equation}
Take some $\omega$ in the set within the probability measure in (\ref{eqn:prod}) i.e. such that there exists $i:=i(\omega)$ and $j:=j(\omega)$ such that $\vert X_i(\omega) Y_i(\omega)-X_j(\omega)Y_j(\omega)\vert\geq \varepsilon$ and also $\omega \in R\cap Q$, and thus
\[
\varepsilon\leq \sigma(\lambda/4)\vert X_i(\omega)-X_j(\omega)|+\rho(\lambda/4) \vert Y_i(\omega)-Y_j(\omega)|
\]
which implies that $\sigma(\lambda/4)\vert X_i(\omega)-X_j(\omega)|\geq \varepsilon/2$ or $\rho(\lambda/4) \vert Y_i(\omega)-Y_j(\omega)|\geq \varepsilon/2$. Therefore if (\ref{eqn:prod}) fails for all $n\leq \chi(\lambda,\varepsilon)$, then we have either
\[
\PP\left(\exists i,j\in [a_n;b_n]\left(|X_i-X_j|\geq \frac{\varepsilon}{2\sigma(\lambda/4)}\right)\right)\geq \frac{\lambda}{4}
\]
or
\[
\PP\left(\exists i,j\in [a_n;b_n]\left(|Y_i-Y_j|\geq \frac{\varepsilon}{2\rho(\lambda/4)}\right)\right)\geq \frac{\lambda}{4}
\]
for each $n\leq \chi(\lambda,\varepsilon)$. As a consequence, there exists either a subsequence $a_{n_0}<b_{n_0}\leq a_{n_1}<b_{n_1}\leq \ldots$ such that
\[
\forall k\leq \phi\left(\frac{\lambda}{4},\frac{\varepsilon}{2\sigma(\lambda/4)}\right)\left[\PP\left(\exists i,j\in [a_{n_k};b_{n_k}]\left(|X_i-X_j|\geq \frac{\varepsilon}{2\sigma(\lambda/4)}\right)\right)\geq \frac{\lambda}{4}\right]
\]
or a subsequence $a_{m_0}<b_{m_0}\leq a_{m_1}<b_{m_1}\leq \ldots$ such that
\[
\forall k\leq \psi\left(\frac{\lambda}{4},\frac{\varepsilon}{2\rho(\lambda/4)}\right)\left[\PP\left(\exists i,j\in [a_{m_k};b_{m_k}]\left(|Y_i-Y_j|\geq \frac{\varepsilon}{2\rho(\lambda/4)}\right)\right)\geq \frac{\lambda}{4}\right]
\]
contradicting the defining properties of $\phi$ or $\psi$.
\end{proof}

\begin{lemma}
\label{res:addtimes:boundedness}
Suppose that $\seq{X_n}$ and $\seq{Y_n}$ have moduli of uniform boundedness $\rho$ and $\sigma$, respectively. Then
\begin{enumerate}[(a)]

	\item a modulus of uniform boundedness for $\seq{X_n+Y_n}$ is given by $$\tau(\lambda):=\rho\left(\lambda/2\right)+\sigma\left(\lambda/2\right)$$
	
	\item a modulus of uniform boundedness for $\seq{X_nY_n}$ is given by $$\tau(\lambda):=\rho\left(\lambda/2\right)\cdot\sigma\left(\lambda/2\right)$$

\end{enumerate}
\end{lemma}

We omit the proof, which follows directly from the definitions.

\section{The quantitative Robbins-Siegmund theorem}
\label{sec:RS:main}

We now state and prove our main theorem. We require the following metastable rate of convergence for supermartingales, which falls out of the general framework of \cite{neri-powell:25:martingale}:
\begin{theorem}[\cite{neri-powell:25:martingale}]
\label{res:supermartingale}
Let $\seq{U_n}$ be a nonnegative supermartingale with $\EE[U_0]<K$ for some $K>1$. Then $\seq{U_n}$ has learnable rate of uniform convergence given by
\begin{equation*}
\phi(\lambda,\varepsilon):=c\left(\frac{K}{\lambda \varepsilon}\right)^2
\end{equation*}
for a universal constant $c\leq 200$.
\end{theorem}

This follows by adapting \cite[Theorem 7.2]{neri-powell:25:martingale}, and a proof in given in Appendix \ref{app:supermartingale}.

\begin{theorem}
[Quantitative Robbins-Siegmund theorem]
\label{res:RS}
Let $\seq{X_n}$, $\seq{A_n}$, $\seq{B_n}$, $\seq{C_n}$ be nonnegative integrable stochastic processes adapted to some filtration $\seq{\mathcal{F}_n}$ such that
\[
\EE[X_{n+1}\mid \mathcal{F}_n]\leq (1+A_n)X_n-B_n+C_n
\]
almost surely for all $n\in\NN$. Suppose that $K>\EE[X_0]$ for $K>1$ and that $\rho,\sigma:(0,1)\to [1,\infty)$ are nonincreasing and satisfy
\[
\PP\left(\prod_{i=0}^\infty (1+A_i)\geq \rho(\lambda)\right)<\lambda \ \ \ \mbox{and} \ \ \ \PP\left(\sum_{i=0}^\infty C_i\geq \sigma(\lambda)\right)<\lambda
\]
for all $\lambda\in (0,1)$. Then $\seq{X_n}$ converges almost surely, with learnable rate of uniform convergence
\[
\phi(\lambda,\varepsilon):=\bar{c}\cdot\left(\frac{\rho\left(\frac{\lambda}{8}\right)\cdot \left(K+\sigma\left(\frac{\lambda}{16}\right)\right)}{\lambda\varepsilon}\right)^2
\]
where $\bar{c}>0$ is a constant that can be computed from that in Theorem \ref{res:supermartingale}, and $\sum_{i=0}^\infty B_i<\infty$ almost surely, with modulus of uniform boundedness
\[
\chi(\lambda):=\frac{10\cdot \rho\left(\frac{\lambda}{2}\right)\cdot \left(K+\sigma\left(\frac{\lambda}{8}\right)\right)}{\lambda}.
\]
\end{theorem}

\begin{proof}
Our strategy is to analyse the standard proof of the result (as in e.g. \cite{robbins-siegmund:71:lemma}) making use of the quantitative lemmas we have already established. First of all, we define
\[
P_n:=\prod_{i=0}^{n-1}(1+A_i), \ \ \ \tilde X_n:=\frac{X_n}{P_n}, \ \ \ \tilde B_n:=\frac{B_n}{P_{n+1}}, \ \ \ \tilde C_n:=\frac{C_n}{P_{n+1}}
\]
with $P_0=1$. Since $\seq{B_n}$ is nonnegative we have
\begin{equation*}
\begin{aligned}
\EE[X_{n+1}\mid \mathcal F_n]\leq (1+A_n)X_n+C_n
\end{aligned}
\end{equation*}
almost surely, and it is easy to show that the sequence $(U_n)$ defined by
\[
U_n:=\tilde X_n-\sum_{i=0}^{n-1} \tilde C_i
\]
with $U_0:=\tilde X_0$ is a supermartingale. For $x>0$ define the stopping time $T_x\in \NN\cup\{\infty\}$ by
\[
T_x:=\inf\left\{n\, : \, \sum_{i=0}^n \tilde C_i>x\right\}.
\]
It is a standard fact that the stopped process $\seq{U_{n\wedge T_x}}$ (where $n\wedge m:=\min\{n,m\}$) is also a supermartingale, and therefore $\seq{U_{n\wedge T_x}+x}$ is a \emph{nonnegative} supermartingale, since on $\{T_x<\infty\}$ we have
\begin{equation*}
\begin{aligned}
U_{n\wedge T_x}+x=\tilde X_{n\wedge T_x}-\sum_{i=0}^{n\wedge T_x-1}\tilde C_i+x\geq \tilde X_{n\wedge T_x}-\sum_{i=0}^{T_x-1}\tilde C_i+x\geq \tilde X_{n\wedge T_x}\geq 0
\end{aligned}
\end{equation*}
and on $\{T_x=\infty\}$ we have $\sum_{i=0}^{n-1}\tilde C_i\leq x$ for all $n\in\NN$, and so also $U_{n\wedge T_x}+x\geq 0$. Noting that $\EE[U_{0\wedge T_x}+x]=\EE[\tilde X_0]+x<K+x$, by Theorem \ref{res:supermartingale} a learnable rate of uniform convergence for $\seq{U_{n\wedge T_x}+x}$ is given by
\[
\phi_1^x(\lambda,\varepsilon):=c\left(\frac{K+x}{\lambda\varepsilon}\right)^2
\]
where $c>0$ is the constant from that theorem. Clearly, this must then also be a learnable rate of uniform convergence for $\seq{U_{n\wedge T_x}}$. Our next step is to find a rate for the unstopped sequence $\seq{U_n}$. To this end, for $\lambda\in (0,1)$ define the event
\[
Q_\lambda:=\sum_{i=0}^\infty \tilde C_i<\sigma\left(\frac{\lambda}{2}\right)
\]
noting that since $\tilde C_i\leq C_i$ we have $\PP(Q_\lambda^c)<\lambda/2$. We now define
\[
\phi_1(\lambda,\varepsilon):=\phi^{\sigma(\lambda/2)}_1(\lambda/2,\varepsilon)=4c\left(\frac{K+\sigma(\lambda/2)}{\lambda\varepsilon}\right)^2
\] 
Then for any $\lambda,\varepsilon\in (0,1)$ and $a_0<b_0\leq a_1<b_1\leq \ldots$ there exists some $n\leq \phi_1(\lambda,\varepsilon)$ such that
\[
\PP\left(\exists i,j\in [a_n;b_n]\left(\vert U_{i\wedge T_{\sigma(\lambda/2)}}-U_{j\wedge T_{\sigma(\lambda/2)}}\vert \geq\varepsilon\right)\right)<\lambda/2
\]
Now, we have
\[
\PP\left(\exists i,j\in [a_n;b_n]\left(\vert U_{i}-U_{j}\vert \geq\varepsilon\right)\right)\leq \PP\left(\exists i,j\in [a_n;b_n]\left(\vert U_{i}-U_{j}\vert \geq\varepsilon\right)\cap Q_{\lambda}\right)+\lambda/2
\]
and if $\omega\in Q_\lambda$ then $T_{\sigma(\lambda/2)}(\omega)=\infty$ and thus $U_{n\wedge T_{\sigma(\lambda/2)}(\omega)}(\omega)=U_n(\omega)$ for all $n\in\NN$, and hence
\begin{equation*}
\begin{aligned}
&\PP\left(\exists i,j\in [a_n;b_n]\left(\vert U_{i}-U_{j}\vert \geq\varepsilon\right)\cap Q_{\lambda}\right)\\
&\leq \PP\left(\exists i,j\in [a_n;b_n]\left(\vert U_{i\wedge T_{\sigma(\lambda/2)}}-U_{j\wedge T_{\sigma(\lambda/2)}}\vert \geq\varepsilon\right)\right)<\lambda/2
\end{aligned}
\end{equation*}
from which it follows that $\phi_1$ is a learnable rate of uniform convergence for $\seq{U_n}$. A similar argument allows us to exhibit a rate of uniform boundedness for $\seq{U_n}$: By Ville's inequality, a modulus of uniform boundedness for the positive supermartingale $\seq{U_{n\wedge T_x}+x}$ is given by
\[
\chi_1^x(\lambda):=\frac{K+x}{\lambda}
\]
It follows that
\[
\PP\left(\sup_{n\in\NN}\vert U_{n\wedge T_x}\vert \geq \frac{2(K+x)}{\lambda}\right)\leq \PP\left(\sup_{n\in\NN}\left(U_{n\wedge T_x}+x\right)\geq \frac{2(K+x)}{\lambda}\right)<\frac{\lambda}{2}
\]
where for the first inequality we note that for any $y> x$, if $|U_{n\wedge T_x(\omega)}(\omega)|\geq y$ then $|U_{n\wedge T_x(\omega)}(\omega)|=U_{n\wedge T_x(\omega)}(\omega)$ since $-U_{n\wedge T_x(\omega)}(\omega)\leq x$, and so in particular $U_{n\wedge T_x(\omega)}(\omega)+x\geq U_{n\wedge T_x(\omega)}(\omega)\geq y$, and since $2(K+x)/\lambda>x$ the inequality on suprema holds. Finally, we have
\begin{equation*}
\begin{aligned}
\PP\left(\sup_{n\in\NN}|U_n|\geq \frac{2(K+\sigma(\lambda/2))}{\lambda}\right)&\leq \PP\left(\left[\sup_{n\in\NN}|U_n|\geq \frac{2(K+\sigma(\lambda/2))}{\lambda}\right]\cap Q_\lambda\right)+\frac{\lambda}{2}\\
&\leq\PP\left(\sup_{n\in\NN}|U_{n\wedge T_{\sigma(\lambda/2)}}|\geq \frac{2(K+\sigma(\lambda/2))}{\lambda}\right)+\frac{\lambda}{2}<\lambda
\end{aligned}
\end{equation*}
and so
\[
\chi_1(\lambda):=\frac{2(K+\sigma(\lambda/2))}{\lambda}
\]
is a modulus of uniform boundedness for $\seq{U_n}$. We now conclude the proof using repeated applications of Lemmas \ref{res:addtimes:stochastic} and \ref{res:addtimes:boundedness}.

First, since $\tilde X_n=U_n+\sum_{i=0}^{n-1}\tilde C_i$, and by $\tilde C_i\leq C_i$, $\sigma$ is also a modulus of uniform boundedness for $\seq{\sum_{i=0}^{n-1}\tilde C_n}$, by Lemma \ref{res:addtimes:boundedness} (a) a modulus of uniform boundedness for $\seq{\tilde X_n}$ is given by
\[
\chi_2(\lambda):=\frac{4(K+\sigma(\lambda/4))}{\lambda}+\sigma(\lambda/2)
\]
and using Lemma \ref{res:sums:converge} (b) in addition to Lemma \ref{res:addtimes:stochastic} (a) a learnable rate of uniform convergence for $\seq{\tilde X_n}$ is given by
\[
\phi_2(\lambda,\varepsilon):=64c\left(\frac{K+\sigma(\lambda/4)}{\lambda\varepsilon}\right)^2+\frac{8\sigma(\lambda/4)}{\lambda\varepsilon}
\] 
For the final step, since $X_n=\tilde X_nP_n$ and $\rho$ is by definition a modulus of uniform boundedness for $\seq{P_n}$, by Lemma \ref{res:monotone:stochastic} and monotonicity of $\seq{P_n}$ a learnable rate of uniform convergence for the sequence is given by $2\rho(\lambda/2)/\lambda\varepsilon$, by Lemma \ref{res:addtimes:stochastic} (b) a learnable rate of uniform convergence for $\seq{X_n}$ is given by any bound on
\[
\phi_2\left(\frac{\lambda}{4},\frac{\varepsilon}{2\rho(\lambda/4)}\right)+\frac{16\cdot\chi_2(\lambda/4)\cdot\rho(\lambda/8)}{\lambda\varepsilon}
\]
The simplified bound in the statement of the theorem then follows in a crude way by bringing together terms, and using the assumptions $\lambda,\varepsilon\in (0,1)$ and that $\rho,\sigma$ are nonincreasing with $\rho(\lambda),\sigma(\lambda)\geq 1$. To obtain the modulus of uniform boundedness on $\sum_{i=0}^\infty B_i$, we define
\[
V_n:=\tilde X_n-\sum_{i=0}^{n-1}(\tilde C_i-\tilde B_i)=U_n+\sum_{i=0}^{n-1} \tilde B_i
\]
with $V_0:=\tilde X_0$. By an essentially identical argument to that for $\seq{U_n}$, we can show that $\seq{V_n}$ is a supermartingale, and defining the stopping time $T_x$ just as before and observing that because $\seq{B_n}$ is nonnegative is it also the case that
\[
V_{n\wedge T_x}+x=U_{n\wedge T_x}+\sum_{i=0}^{n\wedge T_x-1}\tilde B_i+x\geq U_{n\wedge T_x}+x\geq 0
\]
on $\Omega$, it follows that $\seq{V_{n\wedge T_x}+x}$ is a nonnegative supermartingale with $\EE[V_{0\wedge T_x}+x]<K+x$ with modulus of uniform boundedness $\chi_1^x=(K+x)/\lambda$ as defined above. Then just as before, for $\omega\in Q_\lambda$ we have $T_x(\omega)=\infty$ and thus $V_{n\wedge T_x(\omega)}(\omega)=V_n(\omega)$, and so by an identical argument, $\chi_1$ as defined above is a modulus of uniform boundedness for $\seq{V_n}$. We now apply Lemma \ref{res:addtimes:boundedness} several times: Since 
\[
\sum_{i=0}^{n-1}\tilde B_i=V_n+(-U_n)\leq V_n+\sum_{i=0}^{n-1} C_i 
\]
then
\[
\chi_3(\lambda):=\frac{5(K+\sigma(\lambda/4))}{\lambda}\geq \chi_1(\lambda/2)+\sigma(\lambda/2)
\]
is a modulus of uniform boundedness for $\seq{\sum_{i=0}^{n-1}\tilde B_i}$. Finally, since
\[
\sum_{i=0}^{n-1} B_i=\sum_{i=0}^{n-1} \tilde B_iP_{i+1}\leq P_n\sum_{i=0}^{n-1}\tilde B_i
\]
a modulus of uniform boundedness for $\seq{\sum_{i=0}^{n-1}B_i}$ is given by Lemma \ref{res:addtimes:boundedness} as
\[
\chi(\lambda):=\rho(\lambda/2)\cdot \chi_3(\lambda/2)
\]
Making this explicit yields the second part of the theorem. 
\end{proof}

As a byproduct of the proof above, we obtain a modulus of uniform boundedness for $\seq{X_n}$.

\begin{corollary}
\label{cor:RS:bounded}
Under the assumptions of Theorem \ref{res:RS}, a modulus of uniform boundedness for $\seq{X_n}$ is given by
\[
\tau(\lambda):=\frac{9\left(K+\sigma\left(\tfrac{\lambda}{8}\right)\right)\cdot \rho\left(\tfrac{\lambda}{2}\right)}{\lambda}
\]
\end{corollary}

\begin{proof}
Letting $\chi_2$ and $\rho$ be the moduli of uniform boundedness for $\seq{\tilde X_n}$ and $\seq{P_n}$ as defined in the proof of Theorem \ref{res:RS}, by Lemma \ref{res:addtimes:boundedness}, $\chi_2(\lambda/2)\cdot\rho(\lambda/2)$ is a modulus of uniform boundedness for $X_n=\tilde X_nP_n$, and the result follows by substituting in the definition of $\chi_2$ and observing that $\tau(\lambda)\geq \chi_2(\lambda/2)\cdot\rho(\lambda/2)$.
\end{proof}

We conclude by noting that the nonstochastic counterpart of Theorem \ref{res:RS:original} (given as Lemma 5.31 of \cite{bauschke-combettes:17:book}) is of interest in its own right - see in particular \cite[Section 3]{franci-grammatico:convergence:survey:22} or the recent \cite{arakcheev-bauschke:pp:opial}. A quantitative version of this nonstochastic Robbins-Siegmund\footnote{The special case $\beta_n=0$ corresponds to a well-known lemma of Qihou \cite[Lemma 2]{qihou:01:nonexpansive}, which has already been analysed by Kohlenbach and Lambov in \cite[Lemma 16]{kohlenbach-lambov:04:asymptotically:nonexpansive}.} theorem follows immediately from Theorem \ref{res:RS}, though simplifying the analysis results in the following improved bounds:
\begin{theorem}
\label{res:RS:nonstochastic}
Let $\seq{x_n}$, $\seq{\alpha_n}$, $\seq{\beta_n}$ and $\seq{\gamma_n}$ be sequences of nonnegative reals with
\[
x_{n+1}\leq (1+\alpha_n)x_n-\beta_n+\gamma_n
\]
for all $n\in\NN$. Suppose that $K,L,M>0$ satisfy $x_0<K$, $\prod_{i=0}^\infty (1+\alpha_i)<L$ and $\sum_{i=0}^\infty \gamma_i<M$. Then $\seq{x_n}$ converges with learnable rate
\[
\phi(\varepsilon):=\frac{8L(K+M)}{\varepsilon}
\]
and moreover $\sum_{i=0}^\infty \beta_i<L(K+M)$.
\end{theorem}

\begin{proof}
An analogous but much simpler argument to the proof of Theorem \ref{res:RS}, where instead of Theorem \ref{res:supermartingale} we use the standard fact (cf. \cite[Chapter 2]{kohlenbach:08:book}) that whenever $(u_n)$ is a nonincreasing sequence of positive reals with $u_0<K$, a learnable rate of convergence for $(u_n)$ is given by $\phi(\varepsilon):=K/\varepsilon$.
\end{proof}

\section{Convergence of abstract stochastic algorithms}
\label{sec:RS:app}

The purpose of this final section is to illustrate how the different components of Theorem \ref{res:RS} can be combined with slightly stronger assumptions to produce new abstract convergence results that resemble specific strategies for using the Robbins-Siegmund theorem. We anticipate that much is possible in this general direction, and our aim here is simply to discuss one of the most basic and widely used strategies for applying the Robbins-Siegmund theorem, where both parts of the theorem are used together to prove the stronger statement $X_n\to 0$. Typically, $X_n$ is related to $\norm{x_n-\theta}$ for some stochastic algorithm $\seq{x_n}$ intended to approximate an object $\theta$, where from $X_n\to 0$ we are then able to infer that $x_n\to \theta$. Classic examples of algorithms whose convergence can be proven using this strategy include the Robbins-Monro \cite{robbins-monro:51:stochastic} and Kiefer-Wolfowitz \cite{kiefer-wolfowitz:52:scheme} schemes, and more modern examples number in the hundreds, including, for example, online learning algorithms \cite{Bott1998}, block-coordinate methods \cite{combettes-pesquet:15:fejer}, general gradient methods as captured recently by the unifying framework of \cite{LM2022}, along with algorithms in more exotic settings, such as Hadamard spaces \cite{Bac2014} and Riemannian manifolds \cite{TFBJ2018}, to name just a few (see \cite{franci-grammatico:convergence:survey:22} for a comprehensive survey of the many ways in which the Robbins-Siegmund theorem and similar abstract convergence results have been applied, and \cite{neri-powell:23:recineq} for a related survey on how logical methods can be applied in this general context).

To represent such an approach in the abstract, throughout this section we consider the special case of Theorem \ref{res:RS} where $B_n:=u_nV_n$ for $\seq{V_n}$ some nonnegative integrable stochastic process and $\seq{u_n}$ a sequence of nonnegative reals -- in other words, we assume that
\begin{equation}
\EE[X_{n+1}\mid\mathcal{F}_n]\leq (1+A_n)X_n-u_nV_n+C_n
\label{RM}
\end{equation}
almost surely for all $n\in\NN$ and that the other conditions of Theorem \ref{res:RS} are satisfied with respect to some fixed $K>\EE[X_0]$ and moduli of uniform boundedness $\rho,\sigma$ for $\prod(1+A_i),\sum C_i<\infty$ respectively. We denote the resulting learnable rate of convergence for $\seq{X_n}$ by $\phi$ and modulus of uniform boundedness for $\sum u_iV_i$ by $\chi$, both as defined in Theorem \ref{res:RS}, and the modulus of uniform boundedness for $\seq{X_n}$ by $\tau$, as in Corollary \ref{cor:RS:bounded}.

Intuitively, the $\seq{u_n}$ perform the role of \emph{step-sizes}, and for that reason we also require the classic assumption that $\sum u_i=\infty$. Because we are concerned with quantitative results, we require in addition a so-called rate of divergence for this series, which here will be a function $r:\NN\times (0,\infty)\to \NN$ such that
\[
\forall n,x\left(\sum_{i=n}^{n+r(n,x)} u_i\geq x\right).
\]
An immediate corollary of the Robbins-Siegmund theorem in this case is $\liminf_{n\to\infty} V_n=0$. We make this property quantitative by introducing a corresponding modulus:
\begin{definition}
\label{def:liminf}
A function $\Phi:(0,1)\times (0,1)\times \NN\to\NN$ is a (stochastic) liminf-modulus for a sequence of random variables $\seq{V_n}$ if for all $\lambda,\varepsilon>0$ and $n\in\NN$:
\[
\PP\left(\forall k\in [n;n+\Phi(\lambda,\varepsilon,n)](V_k\geq\varepsilon)\right)<\lambda.
\] 
\end{definition}
The following is now straightforward:
\begin{corollary}
\label{res:div:to:conv}
Under the above conditions, a liminf-modulus for $\seq{V_n}$ is given by
\[
\Phi(\lambda,\varepsilon,n):=r\left(n,\frac{\chi(\lambda)}{\varepsilon}\right).
\]
\end{corollary}

\begin{proof}
Fix $\lambda,\varepsilon>0$ and $n\in\NN$, and take some $\omega\in \Omega$. Suppose that $V_i(\omega)\geq\varepsilon$ for all $k\in [n;n+\Phi(n,\varepsilon)]$. Then
\[
\chi(\lambda)\leq \varepsilon\sum_{i=n}^{n+\Phi(\lambda,\varepsilon,n)} u_i\leq \sum_{i=n}^{n+\Phi(\lambda,\varepsilon,n)} u_iV_i(\omega)\leq \sum_{i=0}^{\infty} u_iV_i(\omega)
\]
and therefore we have shown that
\[
\PP\left(\forall k\in [n;n+\Phi(\lambda,\varepsilon,n)](V_k\geq\varepsilon)\right)\leq \PP\left(\sum_{i=0}^\infty u_iV_i\geq \chi(\lambda)\right)<\lambda
\]
which completes the proof.
\end{proof}

In the particular situations which we seek to present in the abstract, there would be a concrete relationship between $\seq{V_n}$ and $\seq{X_n}$ which allows us to infer $\liminf_{n\to\infty}X_n$ from $\liminf_{n\to\infty} V_n$. For example, in the well-known application of the Robbins-Siegmund theorem to prove convergence of the Robbins-Monro algorithm, following the presentation of \cite{lai:03:survey} (cf. page 393) the main recurrence has the form (\ref{RM}) for $X_n:=(x_{n+1}-\theta)^2$ and $V_n:=M(x_{n+1})(x_{n+1}-\theta)$, where $M$ is a measurable function satisfying
\[
\inf_{\varepsilon\leq |x-\theta|\leq \varepsilon^{-1}}\left\{M(x)(x-\theta)\right\}>0.
\]
Similarly, the more recent quantitative analysis of stochastic gradient descent \cite[Theorem 1]{karandikar:pp:stochastic} asserts the existence of a function $\eta:(0,\infty)\to (0,\infty)$ satisfying 
\[
V_n(\omega)\geq \eta(X_n(\omega))
\]
where $\eta$ has the additional property that for any $0<\varepsilon\leq M$
\[
\inf_{\varepsilon\leq r\leq K}\eta(r)>0.
\]
We can give such relationships a general quantitative form, which also facilitates a quantitative route between the corresponding liminf-moduli, in the following way:

\begin{lemma}
\label{vn:xn}
Suppose that $\delta:(0,\infty)\times (0,\infty)\to (0,\infty)$ satisfies
\begin{equation}
\varepsilon\leq X_n(\omega)\leq K\implies V_n(\omega)\geq \delta(\varepsilon,K)
\label{vnxn}
\end{equation}
for all $0<\varepsilon\leq K$ and $\omega\in\Omega$. Then a liminf-modulus for $\seq{X_n}$ is given by
\[
\Psi(\lambda,\varepsilon,n):=\Phi\left(\lambda/2,\delta\left(\varepsilon,\tau(\lambda/2)\right),n\right)
\]
where $\Phi$ is any liminf-modulus for $\seq{V_n}$, and so in particular that defined in Corollary \ref{res:div:to:conv}.
\end{lemma}

\begin{proof}
Fix $\lambda,\varepsilon>0$ and $n$ and define the events
	\[
	A:=\sup_{k\in\NN} X_k< \tau(\lambda/2) \ \ \ \mbox{and} \ \ \ B:=\forall k\in [n;n+\Psi(\lambda,\varepsilon,n)](X_k\geq \varepsilon)
	\]
where we recall that $\tau$ is a modulus of uniform boundedness for $\seq{X_n}$.	Take $\omega\in A\cap B$. Then for any $k\in [n;n+\Psi(\lambda,\varepsilon,n)]$ we have $\varepsilon\leq X_k(\omega)< \tau(\lambda/2)$ and thus
	\[
	V_k(\omega)\geq \delta\left(\varepsilon,\tau(\lambda/2)\right)
	\]
	We have therefore shown that
	\[
	\PP(A\cap B)\leq\PP(\forall k\in [n;n+\Psi(\lambda,\varepsilon,n)](V_k\geq \delta\left(\varepsilon,\tau(\lambda/2)\right))<\lambda/2
	\]
	and therefore
	\[
	\PP(B)\leq \PP(A\cap B)+\PP(A^c)< \lambda/2+\lambda/2=\lambda
	\]
	and we are done.
\end{proof}
To give a concrete example of a modulus $\delta$ as in Theorem \ref{vn:xn}, in the case of the Robbins-Monro algorithm discussed above, for any function $\mu:(0,\infty)\to (0,\infty)$ satisfying
\[
\inf_{\varepsilon\leq |x-\theta|\leq \varepsilon^{-1}}\left\{M(x)(x-\theta)\right\}\geq \mu(\varepsilon)
\]
the function
\[
\delta(\varepsilon,K):=\mu\left(\sqrt{\min\left\{\varepsilon,1/K\right\}}\right)
\]
would satisfy (\ref{vnxn}). Putting everything together now gives us the following main convergence result, which can be seen as a special case of Theorem \ref{res:RS} for abstract Robbins-Monro-like algorithms.
\begin{theorem}
\label{res:RM}
Suppose that $\seq{X_n},\seq{V_n},\seq{A_n},\seq{C_n}$ and $\seq{u_n}$ satisfy (\ref{RM}), that $K>\EE[X_0]$ and that $\rho,\sigma$ are nonincreasing moduli of uniform boundedness for $\prod(1+A_i),\sum C_i<\infty$ respectively. Suppose that $\sum u_i=\infty$ with rate $r$, and that (\ref{vnxn}) is satisfied for some function $\delta$. Then $X_n\to 0$ almost surely, and moreover for any $\lambda,\varepsilon>0$ and $g:\NN\to\NN$ there exists some $n\leq \Gamma(\lambda,\varepsilon,g)$ such that
\[
\PP\left(\exists k\in [n;n+g(n)](X_k\geq\varepsilon)\right)<\lambda
\]
where
\[
\Gamma(\lambda,\varepsilon,g):=\tilde f^{(\lceil\phi(\lambda/2,\varepsilon/2)\rceil)}(0) \ \ \ \mbox{for} \ \ \ f(j):=\max\{g(j),\Psi(\lambda/2,\varepsilon/2,j)\}
\]
and $\tilde f(j):=j+f(j)$, and where $\Psi$ is defined as in Lemma \ref{vn:xn} for $\Phi$ as defined in Corollary \ref{res:div:to:conv} and $\phi,\chi$ and $\tau$ as defined in Theorem \ref{res:RS} and Corollary \ref{cor:RS:bounded}.
\end{theorem}

\begin{proof}
Recall that $\phi$ is a learnable rate of uniform convergence for $\seq{X_n}$, and so now fixing $\lambda,\varepsilon>0$, there exists some $n\leq \tilde f^{(\lceil\phi(\lambda/2,\varepsilon/2)\rceil)}(0)$ such that
\[
\PP\left(\exists i,j\in [n;n+f(n)](|X_i-X_j|\geq \varepsilon/2)\right)<\lambda/2
\]
Let $D$ be the event inside the probability above. For this particular $n$, since $\Psi$ is a liminf modulus for $\seq{X_n}$ we also have
\[
\PP\left(\forall k\in [n;n+\Psi(\lambda/2,\varepsilon/2,n)](X_k\geq \varepsilon/2)\right)<\lambda/2
\]
Let $E$ be the event inside this probability. Fix $\omega\in\Omega$ and suppose that there exists some $k(\omega)\in [n;n+g(n)]\subseteq [n;n+f(n)]$ such that $X_{k(\omega)}(\omega)\geq \varepsilon$. Either $\omega\in E$, or there exists some $j(\omega)\in [n;n+\Psi(\lambda/2,\varepsilon/2,n)]\subseteq [n;n+f(n)]$ such that $X_{j(\omega)}(\omega)<\varepsilon/2$. But this then implies that $|X_{j(\omega)}(\omega)-X_{k(\omega)}(\omega)|\geq\varepsilon/2$, and so $\omega \in D$. Therefore
\[
\PP\left(\exists k\in [n;n+g(n)](X_k\geq \varepsilon)\right)\leq \PP(D\cup E)\leq \PP(D)+\PP(E)<\lambda
\]
and the theorem is proved.
\end{proof}

As expected, the conclusion of Theorem \ref{res:RM} does not provide a direct rate of convergence for $X_n\to 0$, only a metastable version (though direct rates are possible under stronger assumptions, as explored recently in \cite[Section 3]{neri-pischke-powell:pp:fejer}). However, it should be noted that even the weaker liminf-modulus $\Psi$ for $\seq{X_n}$ constructed in Lemma \ref{vn:xn} provides relevant numerical information on the convergence speed of the underlying algorithm: In particular, defining $f(\lambda,\varepsilon):=\Psi(\lambda,\varepsilon,0)$ we have
\[
\PP\left(\exists k\leq f(\lambda,\varepsilon)(X_k<\varepsilon)\right)>1-\lambda.
\]
In a typical concrete scenario where we might have e.g. $X_n:=\norm{x_n-\theta}$ for some stochastic algorithm $\seq{x_n}$ converging to a solution $\theta$, we have therefore shown that up to any desired degree of accuracy $\lambda>0$ there exists a set $A_\lambda$ with $\PP(A_\lambda^c)<\lambda$ along with a function $\psi_\lambda:=f(\lambda,\cdot)$ such that for any $\omega\in A_\lambda$, the trajectory $\seq{x_n(\omega)}$ produces an approximate $\varepsilon$-solution $x_k(\omega)$ -- in the sense that $\norm{x_k(\omega)-\theta}<\varepsilon$ -- within the first $\psi_\lambda(\varepsilon)$ iterates. In general, quantitative information of this kind has a clear practical significance. 

To give a concrete example of the general constructions above, if the step sizes $u_i:=u$ are constant, then we can define our rate of divergence $r(n,x):=\lceil x/u\rceil$, and so
\[
\Phi(\lambda,\varepsilon,n)=\ceil*{\frac{\chi(\lambda)}{u\cdot\varepsilon}}=\ceil*{\frac{10\cdot\rho\left(\frac{\lambda}{2}\right)\cdot\left(K+\sigma\left(\frac{\lambda}{8}\right)\right)}{u\cdot\lambda\varepsilon}}
\]
which in turn implies that
\[
f(\lambda,\varepsilon):=\Psi(\lambda,\varepsilon,0)=\ceil*{\frac{20\cdot\rho\left(\frac{\lambda}{4}\right)\cdot\left(K+\sigma\left(\frac{\lambda}{16}\right)\right)}{u\cdot\lambda\cdot\delta(\varepsilon,\tau(\lambda/2))}}.
\]
Further assumptions then simplify this function: If, for example, we have $\prod_{i=0}^\infty (1+A_i)<L$ and $\sum_{i=0}^\infty C_i<M$ almost surely, then the corresponding moduli of uniform boundedness are given by constant functions $\rho(\lambda):=L$ and $\sigma(\lambda):=M$. If, furthermore, condition (\ref{vnxn}) is strengthened to
\[
X_n(\omega)\geq\varepsilon\implies V_n(\omega)\geq \delta(\varepsilon)
\]
it follows that
\[
f(\lambda,\varepsilon)\leq \frac{d}{\lambda\cdot\delta(\varepsilon)}
\]
for a suitable constant $d$ definable from $K,L,M$ and $u$, and so in particular a bound on the search for approximate solutions is given by
\[
\PP\left(\exists k\leq \frac{d}{\lambda\cdot\delta(\varepsilon)}\,(X_k<\varepsilon)\right)>1-\lambda.
\]
This demonstrates that the general quantitative results in the main part of this paper can, under additional assumptions, be used to obtain numerical information that is both simple and highly uniform, which would in turn be inherited by any stochastic algorithms whose convergence reduces to that particular instance of the Robbins-Siegmund theorem.   

\bigskip
\noindent\textbf{Acknowledgements.} The authors are indebted to Nicholas Pischke for numerous insightful discussions and comments on the topics of this paper. They would also like the thank the two anonymous referees, particularly for the suggestion to make the article more concise. No data were created during this research.\medskip

\noindent\textbf{Funding.} The first author was partially supported by the EPSRC Centre for Doctoral Training in Digital Entertainment EP/L016540/1, and the second author was partially supported by the EPSRC grant EP/W035847/1.

\appendix

\section{Proof of Theorem \ref{res:supermartingale}}
\label{app:supermartingale}

Though already established in a much more general terms \cite{neri-powell:25:martingale}, in this appendix we prove Theorem \ref{res:supermartingale} from first principles, taking advantage of the lack of generality to give a streamlined proof. We start with a variant of Doob's well-known upcrossing inequalities \cite{doob:53:stochastic} for martingales. Fixing for the remainder of the section a nonnegative supermartingale $\seq{U_n}$, for $0\leq a<b$ let the random variable $C_N[a,b]$ denote the number of times $\seq{U_n}$ crosses the interval $[a,b]$ up to time $N$, and similarly, $D_N[a,b]$ the number of times it \emph{downcrosses} the interval. Since between any pair of downcrossings there must be exactly one upcrossing, we have $C_N[a,b]\leq 2D_N[a,b]+1$. Define $C[a,b]=\lim_{N\to\infty} C_N[a,b]$ and $D[a,b]$ similarly, noting that also $C[a,b]\leq 2D[a,b]+1$.

\begin{lemma}
\label{res:doob}
For $M,p>0$ let $\mathcal{P}(M,p)$ denote the partition of the interval $[0,M]$ into exactly $p$ equal sized closed intervals. Then
\[
C[a,b]\leq \frac{2p\cdot \EE[U_0]}{M}+1
\]
for all $[a,b]\in \mathcal{P}(M,p)$.
\end{lemma}

\begin{proof}
The proof follows the standard pattern for deriving crossing inequalities (see e.g. \cite{doob:53:stochastic} or \cite{williams:91:martingales}). Given $[a,b]\in \mathcal{P}(M,p)$ define the predictable process $Z_n\in\{0,1\}$ as follows: $Z_1=1$ if $U_0>b$ and $Z_1=0$ otherwise, and $Z_{n+1}=1$ if either $Z_n=1$ and $U_n\geq a$, or $Z_n=0$ and $U_n>b$, and if neither hold then $Z_{n+1}=0$. In other words, $\seq{Z_n}$ represents a betting strategy for maximising our losses, where we start betting each time $\seq{U_n}$ rises above $b$, and stop every time it falls below $a$. Our loss at time $N$ is then at least $(b-a)D_N[a,b]$ minus an error term for the current downcrossing not yet completed, or to be more precise:
\begin{equation}
\label{eqn:downcrossings}
\sum_{i=1}^N Z_i(U_i-U_{i-1})\leq -(b-a)D_N[a,b]+(U_N-b)^+
\end{equation}
Now let $Z'_n:=1-Z_n$, which is also a predictable process, and since $\seq{U_n}$ is a supermartingale we have
\[
\EE\left[\sum_{i=1}^N Z'_i(U_i-U_{i-1})\right]\leq 0
\]
for all $n\in\NN$, and thus 
\begin{equation*}
\begin{aligned}
\EE[U_N]-\EE[U_0]&=\EE\left[\sum_{i=1}^N(U_i-U_{i-1})\right]\\
&\leq \EE\left[\sum_{i=1}^N Z_i(U_i-U_{i-1})\right]+\EE\left[\sum_{i=1}^N Z'_i(U_i-U_{i-1})\right]\\
&\leq \EE\left[\sum_{i=1}^N Z_i(U_i-U_{i-1})\right]
\end{aligned}
\end{equation*}
and so substituting the above into (\ref{eqn:downcrossings}) we get
\begin{equation*}
\label{doob:ineq}
(b-a)\EE\left[D_N[a,b]\right]\leq \EE[U_0]-\EE[U_N]+\EE[(U_N-b)^+]\leq \EE[U_0]
\end{equation*}
An alternative derivation of the same inequality for nonnegative supermartingales can be found in \cite{doob:61:notes} (where (\ref{doob:ineq}) above follows immediately from (3.6) of that paper).

Now if $[a,b]\in \mathcal{P}(M,p)$ then $b-a=M/p$ and so
\[
\EE\left[C_N[a,b]\right]\leq 2\EE\left[D_N[a,b]\right]+1\leq \frac{2p\cdot \EE[U_0]}{M}+1
\]
and the result follows.
\end{proof}

\begin{proof}
[Proof of Theorem \ref{res:supermartingale}]
Fix $\lambda,\varepsilon\in (0,1)$ and $a_0<b_0\leq a_1<b_1\leq\ldots$. First note that by Ville's inequality, for $K>\EE[U_0]$ we have
\[
\PP\left(|U_n|\geq \frac{2K}{\lambda}\right)\leq \PP\left(\sup_{n\in\NN}|U_n|\geq \frac{2K}{\lambda}\right)<\frac{\lambda}{2}
\]
and therefore it suffices to find a bound on some $n\in\NN$ satisfying
\begin{equation*}
\PP\left(\exists k,l\in [a_n;b_n]\left(|U_k-U_l|\geq \varepsilon\right) \cap |U_{a_n}|<\frac{2K}{\lambda}\right)<\frac{\lambda}{2}
\end{equation*}
Define $Q_n$ to be the event inside the probability measure above. Divide $[0,2K/\lambda]$ into $p:=\lceil 8K/\lambda\varepsilon\rceil$ subintervals which we label $[\alpha_j,\beta_j]$ for $j=0,\ldots,p-1$, and add a further interval of the same width at the upper end of $[0,2K/\lambda]$, which we label $[\alpha_{p},\beta_{p}]$. Each interval has width $2K/pl\leq \varepsilon/4$.

Now take $\omega\in Q_i$, so that there exists $k(\omega),l(\omega)\in [a_i;b_i]$ with $|U_{k(\omega)}(\omega)-U_{l(\omega)}(\omega)|\geq \varepsilon$ and also $|X_{a_i}|<2K/\lambda$. By the triangle inequality, either $|U_{a_i}(\omega)-U_{k(\omega)}(\omega)|\geq \varepsilon/2$ or $|U_{a_i}(\omega)-U_{l(\omega)}(\omega)|\geq \varepsilon/2$. Since $X_{a_i}\in [0,2K/\lambda]$ and we have our additional interval $[\alpha_{p},\beta_{p}]$, it follows that one of the intervals $[\alpha_j,\beta_j]$ for $j=0,\ldots,p$ is crossed by $\seq{U_n(\omega)}$ somewhere in $[a_i;b_i]$, and therefore defining
\[
T_{i,j}:=\mbox{$\seq{U_n}$ crosses $[\alpha_j,\beta_j]$ somewhere in $[a_i;b_i]$}
\]
we have shown that
\[
Q_i\subseteq \bigcup_{j=0}^p T_{i,j}
\]
Let $s_N:=\sum_{i=0}^N \PP(Q_i)$. Then
\[
s_N\leq \sum_{i=0}^N \PP\left(\bigcup_{j=0}^p T_{i,j}\right)\leq \sum_{i=0}^N \sum_{j=0}^p\PP\left(T_{i,j}\right)
\]
and so there is some $j\in \{0,\ldots,p\}$ such that
\begin{equation*}
\begin{aligned}
&\frac{s_N}{p+1}\leq \sum_{i=0}^N \PP\left(T_{i,j}\right)=\EE\left[\sum_{i=0}^N I_{T_{i,j}}\right]\leq \EE\left[C[\alpha_j,\beta_j]\right]\leq \frac{2p\cdot \EE[U_0]}{2K/\lambda}+1
\end{aligned}
\end{equation*}
where the last inequality follows by Lemma \ref{res:doob}. Therefore since $\EE[U_0]<K$ and $p<9K/\lambda\varepsilon$
\[
s_N\leq (p+1)\left(p\lambda+1\right)<\left(\frac{9K}{\lambda\varepsilon}+1\right)\left(\frac{9K}{\varepsilon}+1\right)<\frac{100 K^2}{\lambda\varepsilon^2}
\]
Finally, then, if $\PP(Q_i)\geq \lambda/2$ for all $i\leq N$ then
\[
\frac{(N+1)\lambda}{2}\leq s_N<\frac{100 K^2}{\lambda\varepsilon^2}
\]
a contradiction for $N:=200K^2/\lambda^2\varepsilon^2$.
\end{proof}

\bibliographystyle{acm} 
\bibliography{tpbiblio}      


\end{document}